\documentclass[10pt,reqno]{amsart}
\usepackage{amssymb}
\usepackage{amsfonts}
\usepackage[all]{xy}
\usepackage[english]{babel}
\usepackage{booktabs,multirow,graphicx}
\usepackage[notref,notcite,final]{showkeys}
\usepackage{enumerate}
\usepackage{verbatim}
\usepackage{tikz} 
\newcommand{\omeg}{\omega}

\theoremstyle{plain}
\newtheorem{theorem}{Theorem}
\newtheorem*{claim}{Claim}
\newtheorem{lemma}{Lemma}
\theoremstyle{definition}
\newtheorem{definition}{Definition}
\newtheorem*{remark}{Remark}

\begin{document}

\title{On the weakest version of distributional chaos}

\author{Jana Hant\'akov\'a}
\author{Samuel Roth*}
\author{Zuzana Roth}

\address{ Mathematical Institute, Silesian University, CZ-746 01 
Opava, Czech Republic}

\email{jana.hantakova@math.slu.cz (Jana Hant\'akov\'a), samuel.roth@math.slu.cz (Samuel Roth), zuzana.roth@math.slu.cz (Zuzana Roth)}

\thanks{* denotes the corresponding author\\ The research was supported by grant SGS/2/2013 from  the Silesian University in Opava. Support of this institution is gratefully acknowledged.}

\begin{abstract} 
The aim of the paper is to correct and improve some results concerning distributional chaos of type 3. We show that in a general compact metric space, distributional chaos of type 3, denoted DC3, even when assuming the existence of an uncountable scrambled set, is a very weak form of chaos. In particular, (i)  the chaos can be unstable (it can be destroyed by conjugacy),  and (ii)  such an unstable system may contain no Li-Yorke pair. However, the definition can be strengthened to get DC$2\frac{1}{2}$  which is a topological invariant and implies Li-Yorke chaos, similarly as types DC1 and DC2; but unlike them, strict DC$2\frac{1}{2}$ systems must have zero topological entropy.\\
{\small {2000 {\it Mathematics Subject Classification.}}
Primary 37D45; 37B40.
}
\end{abstract}

\maketitle
\pagestyle{myheadings}
\markboth{J. Hant\'akov\'a, S. Roth, Z. Roth}
{On the weakest version of distributional chaos}

\section{Introduction}
The study of chaotic pairs in dynamics goes back to Li and Yorke \cite{LY}, who studied pairs of points with the property that their orbits neither approach each other asymptotically, nor do they eventually separate from each other by any fixed positive distance. Schweizer and Smital \cite{SchSm} introduced the related concept of a distributionally chaotic pair, which means, roughly speaking, that the statistical distribution of distances between the orbits does not converge. Distributional chaos was later divided into three types, DC1, DC2, and DC3, see \cite{BSS}.  

The relations between the three versions of distributional chaos, and the relation between distributional chaos and Li-Yorke chaos are investigated by many authors, see e.g. \cite{BSS,WLF,WLW,O}. One can easily see from the definitions that DC1 implies DC2 and DC2 implies DC3. On the other hand, there are examples which show that DC1 is stronger than DC2 and DC2 is stronger than DC3 (see \cite{BSS,WLF}). It is also obvious that either DC1 or DC2 implies Li-Yorke chaos. Moreover, there are Li-Yorke chaotic continuous maps of the interval with zero topological entropy; by \cite{SchSm}, such maps cannot be distributionally chaotic. This shows that Li-Yorke chaos need not imply any of the three versions of distributional chaos.

We focus on the properties of the weakest form of distributional chaos, DC3.  Unlike its stronger relatives, DC3 chaos does not imply Li-Yorke chaos, and DC3 chaos is not an invariant of topological conjugacy.  In a weak sense, these two results were already stated in \cite[Theorems 1 and 2]{BSS}. However, it should be noticed that the distributional chaos in \cite{BSS} was defined as the existence of a single distributionally scrambled pair, but nowadays it is generally assumed that distributional chaos means the existence of an uncountable distributionally scrambled set.  Moreover, the proof of \cite[Theorem 2]{BSS} is unfortunately in error - the authors constructed a conjugacy which destroys a DC3 pair, but they overlooked many other DC3 pairs which persist. Our first goal, then, is to give stronger statements and correct proofs of these two theorems.

Our second goal is to strengthen the definition of the DC3 pair in such a way that it is preserved under conjugacy and implies Li-Yorke chaos -- we denote the new definition by DC2$\frac{1}{2}$. We also provide an example which shows that $DC2\frac{1}{2}$ is essentially weaker than $DC2$. This example possesses no DC2 pair, hence, by results in \cite{DL}, its topological entropy must be zero.

Another strengthened distributional chaos, denoted by DC1$\frac{1}{2}$, was introduced in \cite{DL}. DC1$\frac{1}{2}$ chaos is stronger than $DC2$ and is implied by positive topological entropy.

The paper is organised as follows: the first and second sections are introductory. In the third section we show the error in \cite{BSS} and simultaneously prove that even an uncountable DC3 scrambled set does not imply Li-Yorke chaos. The fourth section proves (with two new examples) that both the existence of a DC3 pair and the existence of an uncountable DC3 scrambled set are not conjugacy invariants. The fifth section introduces our new definition of DC2$\frac{1}{2}$.

\section{Terminology}
Let $(X,d)$ be a non-empty compact metric space.  A pair $(X,f)$, where $f$ is a continuous self-map acting on $X$, is called a \emph{topological dynamical system}.  We define the \emph{forward orbit} of $x$, denoted by $Orb^+_f(x)$ as the set $\{f^n(x):n\geq0\}$.
Let $(X,f)$ and $(Y,g)$ be dynamical systems on compact metric spaces. A continuous map $\pi: X\rightarrow Y$ is a conjugacy between $f$ and $g$ if $\pi$ is one-to-one and onto and $\pi\circ f=g\circ\pi$. Denote by $I$ the unit interval $[0,1]$.

\begin{definition}
A pair of two different points $(x_1, x_2)\in X^2$ is \emph{scrambled} if 
\begin{equation}\label{Li1}\liminf_{k\to\infty}d(f^k(x_1),f^{k}(x_2))=0\end{equation} and 
\begin{equation}\label{Li2}\limsup_{k\to\infty} d(f^k(x_1),f^{k}(x_2))>0.\end{equation}
A subset $S$ of $X$ is \emph{scrambled} if every pair of distinct points in $S$ scrambled. The system $(X,f)$ is \emph{chaotic} if there exists an uncountable scrambled set.
We call a pair $(x_1, x_2)\in X^2$ \emph{proximal} if (\ref{Li1}) holds (otherwise we say $(x_1, x_2)\in X^2$ is \emph{distal}). If (\ref{Li2}) does not hold, i. e. $$\limsup_{k\to\infty} d(f^k(x_i),f^{k}(x_j))=0,$$
we say that $(x_1, x_2)\in X^2$ is \emph{asymptotic}. A pair of points is scrambled simply if it is proximal but not asymptotic. A dynamical system $X$ is distal if every pair of distinct points in $X$ is distal.
\end{definition}

\begin{definition}
For a pair $(x_1, x_2)$ of
points in $X$, define the \emph{lower distribution function} generated by $f$ as
$$\Phi_{(x_1, x_2)}(\delta)=\displaystyle\liminf_{m\to\infty}\frac{1}{m}\#\{0 \le k \le m;d(f^k(x_1),f^{k}(x_2))<\delta\},$$
and the \emph{upper distribution function} as 
$$\Phi^*_{(x_1, x_2)}(\delta)=\displaystyle\limsup_{m\to\infty}\frac{1}{m}\#\{0 \le k \le m;d(f^k(x_1),f^{k}(x_2))<\delta\},$$
where $\#A$ denotes the cardinality of the set $A$.\\ 
A pair $(x_1, x_2)\in X^2$ is called 
\emph{distributionally scrambled of type 1} if 
$$\Phi^*_{(x_1, x_2)}(\delta)=1, \mbox{  for every $0<\delta\le \text{diam }X$}$$
  and  
  $$ \Phi_{(x_1, x_2)}(\epsilon)=0, \mbox{  for some  }0<\epsilon\le \text{diam }X ,$$
\emph{distributionally scrambled of type 2} if 
$$\Phi^*_{(x_1, x_2)}(\delta)=1, \mbox{  for every $0<\delta\le \text{diam }X$}$$
  and  
  $$ \Phi_{(x_1, x_2)}(\epsilon)< 1,\mbox{  for some  }0<\epsilon\le \text{diam }X ,$$
\emph{distributionally scrambled of type 3} if $$ \Phi_{(x_1, x_2)}(\delta)<\Phi^*_{(x_1, x_2)}(\delta), \mbox{  for every $\delta\in (a,b),$ where }0\leq a<b\le \text{diam }X.$$
The dynamical system $(X,f)$ is \emph{distributionally chaotic of type $i$} (DC$i$ for short), where $i=1,2,3$, if there is an uncountable set $S\subset X$ such that any pair of
distinct points from $S$ is distributionally scrambled of type $i$.
\end{definition}

Let $X,Y$ be compact metric spaces and $X\times Y$ be equipped with the product topology. Then a continuous map $F:X\times Y\to X\times Y$ is a \emph{skew-product mapping} if it has the form $F\big((x,y)\big)=\big(f(x),g_x(y)\big)$.  Then $f:X \to X$ is called the \emph{base map} and the maps $g_x:{Y\to Y}$ are called \emph{fiber maps}.

\section{Distal DC3 system}

The main goal of this section is to prove that DC3 chaos does not imply Li-Yorke chaos.  We will prove this statement in the strongest possible sense -- we present a system with an uncountable DC3 scrambled set but without any Li-Yorke pairs.

\begin{theorem}\label{th:DC3LiY}
There exists a distal dynamical system which is DC3 chaotic.  Thus, DC3 chaos does not imply Li-Yorke chaos.
\end{theorem}

Our proof analyzes a system which was constructed in \cite{BSS}, consisting of a continuous map $F$ on a compact space $M$ equipped with two different compatible metrics $\rho, \rho'$.  The authors of \cite{BSS} identify a pair of points $u,v\in M$ which is DC3 scrambled with respect to the metric $\rho$ but not with respect to $\rho'$.  Then the authors claim without proof that the map $F$ has no DC3 pairs with respect to the metric $\rho'$.  We challenge that claim.  We will show instead that $F$ has an uncountable DC3 scrambled set with respect to both metrics $\rho, \rho'$, provided that a certain model parameter is chosen sufficiently large.  This serves two purposes.  First, since $(M,F)$ is a distal dynamical system, it establishes our Theorem \ref{th:DC3LiY}.  Second, it motivates a correct proof that DC3 chaos is not conjugacy invariant, which is the content of Theorems \ref{ThPairs} and \ref{Th3} in the next section.

Let us recall the construction of the map $F$ from \cite{BSS}.  First, fix an arbitrary parameter $\alpha$ with $\frac12<\alpha<1$.  Then choose an increasing sequence of natural numbers $\left\{n_i\right\}_{i=1}^\infty$ such that
\begin{equation}\label{eq:limalpha}
\lim_{i\to\infty}\frac{(2^{n_1}-2)}{2^{n_1}} \cdot \frac{(2^{n_2}-4)}{2^{n_2}} \cdot \frac{(2^{n_3}-8)}{2^{n_3}} \cdot \ldots \cdot \frac{(2^{n_i}-2^i)}{2^{n_i}} = \alpha.
\end{equation}
We will also use the notation
\begin{equation*}
m_i := n_1 + n_2 + \ldots + n_i.
\end{equation*}
For the phase space take the Cartesian product $M=X\times \Omega$ where
\begin{equation*}
X=\left\{\left(\cos 2\pi\theta, \sin 2\pi\theta \right) ; \theta \in [0,1] \right\} \cup \left\{ (2,0) \right\} \subset \mathbb{R}^2, \quad \Omega = \left\{ 0, 1 \right\} ^ \mathbb{N}
\end{equation*}
The space $X$ is equipped with two different metrics.  The first is the standard Euclidean metric $\nu$ inherited from $\mathbb{C}$.  The second metric $\nu'$ is the same as $\nu$, except that $\nu'(x,(2,0))=1$ for any $x\neq(2,0)$; it corresponds to ``moving'' the isolated point of $X$ from $(2,0)$ (outside the unit circle) to $(0,0)$ (the center of the unit circle), so that it is equidistant from all other points.  The space $\Omega$ is equipped with a metric of pointwise convergence $\rho_\Omega$.  The precise choice of $\rho_\Omega$ is not fixed in \cite{BSS}, nor do we fix it here.  We only require that the diameter of $\Omega$ under $\rho_\Omega$ is at most $1$.  Finally, the space $M$ is given the two maximum metrics
\begin{equation*}
\begin{gathered}
\rho\left((x,\omeg),(x',\omeg')\right)=\max\left(\nu(x,x'),\rho_\Omega(\omeg,\omeg')\right)\\
\rho'\left((x,\omeg),(x',\omeg')\right)=\max\left(\nu'(x,x'),\rho_\Omega(\omeg,\omeg')\right).
\end{gathered}
\end{equation*}

Let $\tau:\Omega\to \Omega$ denote the binary adding machine
\begin{equation*}
\tau(\omeg)=\omeg+10000\ldots
\end{equation*}
with ``carrying'' to the right (for details, see for example \cite{Down}).  The map $F:M\to M$ is defined as a skew product.  The base space is $\Omega$ and in the fibers we rotate the unit circle by an angle depending on $\omeg$.  The formula for $F$ is
\begin{equation*}
F( \cos2\pi\theta,\sin2\pi\theta,\omeg)=( \cos2\pi(\theta + p(\omeg)),\sin2\pi(\theta+p(\omeg)), \tau(\omeg)), \quad F(2,0,\omeg) = (2,0,\tau(\omeg))
\end{equation*}
where the rotation angle $p(\omeg)$ is given by the following algorithm.  If $\omeg=1^\infty$, then set $p(\omeg)=0$.  Otherwise, decompose $\omeg$ into the infinite concatenation $\omeg=\omeg^{(1)}\omeg^{(2)}\omeg^{(3)}\cdots$ where each $\omeg^{(i)}$ is a block of length $n_i$.  To be clear, if we write $\omeg=(\omega_i)_{i=1}^\infty$, then $\omeg^{(i)}=\omega_{m_{i-1}+1}\omega_{m_{i-1}+2}\cdots\omega_{m_i}\in\{0,1\}^{n_i}$.  Then define 
\begin{equation*}
k=k(\omeg)=\min\left\{i;\,\omeg^{(i)}\neq1^{n_i}\right\}.
\end{equation*}
Although $k$ is a function of $\omeg$, we will sometimes suppress the argument for simplicity of notation.  By $|\omeg^{(k)}|$ we denote the \emph{evaluation} of the block $\omeg^{(k)}$, where the evaluation operator is defined by
\begin{equation}\label{eq:eval}
|x_1x_2x_3\cdots x_q|=x_1+2x_2+2^2x_3+\cdots+2^{q-1}x_{q}, \qquad q\in\mathbb{N}, x_1,\ldots,x_q\in\{0,1\}.
\end{equation}
Observe that $0\leq |\omeg^{(k)}| < 2^{n_k}-1$, since we know that $\omeg^{(k)}$ consists of $n_k$ symbols, but $\omeg^{(k)}\neq1^{n_k}$ by the definition of $k$.  Finally, set
\begin{equation}\label{eq:p}
p(\omeg) = \begin{cases} 0, & \textnormal{if } 2^{k-1} \leq |\omeg^{(k)}| < 2^{n_k}-2^{k-1}-1 \\ \frac{1}{2^k}, & \textnormal{otherwise} \end{cases}.
\end{equation}

It is proved in \cite{BSS} that the points $u=(-1,0,0^\infty)$, $v=(2,0,0^\infty)$ form a DC3 pair for $F$ with respect to the metric $\rho$ but not $\rho'$, and this proof is correct.  However, we challenge the subsequent claim (asserted without proof in \cite{BSS}) that $F$ has no DC3 pairs with respect to $\rho'$.  We will prove instead the following claim, which establishes Theorem \ref{th:DC3LiY}.

\begin{claim}
Fix $\alpha>\frac34$ and let the map $F:M\to M$ and the metrics $\rho, \rho'$ be as constructed above.  Then there is an uncountable set $S\subset M$ which is pairwise DC3 scrambled with respect to both $\rho$ and $\rho'$.
\end{claim}

Our proof will make use of the map $\sigma:\Omega\to\Omega$ given by the formula
\begin{equation*}
\sigma(\omega_1\omega_2\omega_3\cdots) = 0^{n_1-1}\omega_1 \, 0^{n_2-1}\omega_2 \,0^{n_3-1}\omega_3 \, \cdots.
\end{equation*}

\begin{lemma}\label{lem:freq}
For each $\omeg\in\Omega$ and $i\in\mathbb{N}$ we have the following estimate regarding net rotations along the orbit of $\sigma(\omeg)$
\begin{multline*}
\# \left\{ n\in\{1,\ldots,2^{m_i-2}\} ;\, \sum_{j=0}^{n-1} p(\tau^j\sigma\omeg) = \frac12 \Big((1-\omega_1)+\cdots+(1-\omega_i)\Big) \mod 1\right\} \geq \\
\geq \left(2^{n_i-2}-2^{i-1}-1\right) \cdot \prod_{l=1}^{i-1} \left(2^{n_l}-2^l\right)
\end{multline*}
\end{lemma}

\begin{proof}
We introduce the following functions
\begin{equation}\label{eq:phil}
\varphi_l(n,\sigma\omeg) = \sum_{\substack{j;0\leq j\leq n-1, \\ k(\tau^j\sigma\omeg)=l}} p(\tau^j\sigma\omeg)
\end{equation}
Thus, $\varphi_l$ counts the contribution to the net rotation due to block $l$.  We have
\begin{equation*}
\sum_{j=0}^{n-1} p(\tau^j\sigma\omeg) = \sum_{l=1}^\infty \varphi_l(n,\sigma\omeg), \quad n\in\mathbb{N}.
\end{equation*}
Notice that $\tau^{2^{m_i-2}}$ represents addition by $000\cdots010\cdots$, where the $1$ appears in the next to last coordinate of block $i$.  Thus, if $n\leq2^{m_i-2}$, then for all $j<n$ there is still a zero in the $i$th block of $\tau^j\sigma\omeg$ so that $k(\tau^j\sigma\omeg)\leq i$.  Therefore we only need to add together contributions from blocks $1$ through $i$,
\begin{equation*}
\sum_{j=0}^{n-1} p(\tau^j\sigma\omeg) = \sum_{l=1}^i \varphi_l(n,\sigma\omeg), \quad n\in\{1,\ldots,2^{m_i-2}\}.
\end{equation*}
Comparing definitions \eqref{eq:p} and \eqref{eq:phil}, we find that we can evaluate $\varphi_l(n,\sigma\omeg)$ by looking up the word appearing in the $l$th block of $\tau^n\sigma\omeg$ on one of the tables in Figure \ref{fig:tablephil}.  We will say that the word in block $l$ is a \emph{good word} if $\varphi_l(n,\sigma\omeg)=\frac12(1-\omega_l) \mod 1$.  If $\tau^n\sigma\omeg$ contains good words in each of the blocks $1$ through $i$, then we obtain the desired sum $\sum_{j=0}^{n-1} p(\tau^j\sigma\omeg) = \frac12 ((1-\omega_1)+\cdots+(1-\omega_i)) \mod 1$.  How many ways are there to choose good words in all blocks?

\begin{figure}
\begin{gather*}\def\arraystretch{1.1} \arraycolsep=1em
\begin{array}{cccc}
\multicolumn{4}{c}{\textnormal{\emph{Use this table if }}\omega_l=0} \\
\toprule
\textnormal{Block } l \textnormal{ from } \tau^n\sigma\omeg & \textnormal{Evaluation} & \varphi_l(n,\sigma\omeg) \textnormal{ mod }1 \\
\midrule
00\cdots0 - 000\cdots00 & 0 & 0 & \multirow{5}{*}{\rotatebox[origin=c]{-90}{Bad}} \\
10\cdots0 - 000\cdots00 & 1 & \frac{1}{2^l} \\
01\cdots0 - 000\cdots00 & 2 & \frac{2}{2^l} \\
 \vdots&\vdots&\vdots\\
11\cdots1 - 000\cdots00 & 2^{l-1}-1 & \frac{2^{l-1}-1}{2^l} \\
\midrule
00\cdots0 - 100\cdots00 & 2^l & \frac{1}{2} & \multirow{3}{*}{\rotatebox[origin=c]{-90}{Good}}\\
 \vdots&\vdots&\vdots\\
11\cdots1 - 011\cdots11 & 2^{n_l}-2^{l-1}-1 & \frac{1}{2} \\
\midrule
00\cdots0 - 111\cdots11 & 2^{n_l}-2^{l-1} & \frac{2^{l-1}+1}{2^l} & \multirow{4}{*}{\rotatebox[origin=c]{-90}{Bad}} \\
 \vdots&\vdots&\vdots\\
01\cdots1 - 111\cdots11 & 2^{n_l}-2 & \frac{2^l-1}{2^l} \\
11\cdots1 - 111\cdots11 & 2^{n_l}-1 & 1 \\
\bottomrule
\end{array}
\\[2em]
\def\arraystretch{1.1} \arraycolsep=1em
\begin{array}{cccc}
\multicolumn{4}{c}{\textnormal{\emph{Use this table if }}\omega_l=1} \\
\toprule\textnormal{Block } l \textnormal{ from } \tau^n\sigma\omeg & \textnormal{Evaluation} & \varphi_l(n,\sigma\omeg) \textnormal{ mod }1 \\
\toprule
00\cdots0 - 000\cdots01 & 2^{n_l-1} & 0 & \multirow{4}{*}{\rotatebox[origin=c]{-90}{Good}}\\
10\cdots0 - 000\cdots01 & 2^{n_l-1}+1 & 0 \\
 \vdots&\vdots&\vdots\\
11\cdots1 - 011\cdots11 &  2^{n_l}-2^{l-1}-1 & 0 \\
\midrule
00\cdots0 - 111\cdots11 & 2^{n_l}-2^{l-1} & \frac{1}{2^l} & \multirow{9}{*}{\rotatebox[origin=c]{-90}{Bad}}\\
10\cdots0 - 111\cdots11 & 2^{n_l}-2^{l-1}+1 & \frac{2}{2^l} \\
\vdots & \vdots & \vdots \\
01\cdots1 - 111\cdots11 & 2^{n_l}-2 & \frac{2^{l-1}-1}{2^l} \\
11\cdots1 - 111\cdots11 & 2^{n_l}-1 & \frac{1}{2} \\
00\cdots0 - 000\cdots00 & 0 & \frac{1}{2} \\
10\cdots0 - 000\cdots00 & 1 & \frac{2^{l-1}+1}{2^l} \\
\vdots & \vdots & \vdots \\
11\cdots1 - 000\cdots00 & 2^{l-1}-1 & \frac{2^l-1}{2^l} \\
\midrule
00\cdots0 - 100\cdots00 & 2^{l-1} & 1 & \multirow{4}{*}{\rotatebox[origin=c]{-90}{Good}} \\
10\cdots0 - 100\cdots00 & 2^{l-1}+1 & 1 \\
\vdots & \vdots & \vdots \\
11\cdots1 - 111\cdots10 & 2^{n_l-1}-1 & 1 \\
\bottomrule
\end{array}
\end{gather*}
\caption{Table of values for $\varphi_l(n,\sigma\omeg)$.   Each word appearing here has length $n_l$.  The first $l-1$ digits of the word have been typographically separated from the remaining $n_l-l+1$ digits by a hyphen.}
\label{fig:tablephil}
\end{figure}

Take the set $\left\{\tau^n(\sigma\omeg)\right\}_{n=1}^{2^{m_i-2}}$ and truncate each member of this set to the first $i$ blocks.  We obtain in this way the set of all $x\in\{0,1\}^{m_i}$ in the interval
\begin{equation*}
\overbrace{0\ldots0\omega_1}^\text{Block 1} \;
\overbrace{0\ldots0\omega_2}^\text{Block 2}  \;
\cdots \;
\overbrace{0\ldots00\omega_i}^\text{Block i}
< x \leq 
\overbrace{0\ldots0\omega_1}^\text{Block 1} \;
\overbrace{0\ldots0\omega_2}^\text{Block 2} \;
\cdots \;
\overbrace{0\ldots01\omega_i}^\text{Block i},
\end{equation*}
where the order relation corresponds to the temporal ordering of our orbit segment; we may express this explicitly in terms of the evaluation operator \eqref{eq:eval} with the rule that $x<y$ whenever $|x| < |y|$.  Unfortunately, this interval does not contain arbitrary combinations of words in the first $i$ blocks.  It does, however, contain the subinterval
\begin{equation*}
\overbrace{00\ldots0}^\text{Block 1} \;
\overbrace{00\ldots0}^\text{Block 2} \;
\cdots \;
\overbrace{10\ldots00\omega_i}^\text{Block i}
\leq x \leq 
\overbrace{11\ldots1}^\text{Block 1} \;
\overbrace{11\ldots1}^\text{Block 2} \;
\cdots \;
\overbrace{11\ldots10\omega_i}^\text{Block i}.
\end{equation*}
This interval contains words from the set $\left\{ x\in\{0,1\}^{n_i} ;\, 10\ldots00\omega_i \leq x \leq 11\ldots10\omega_i\right\}$ in block $i$ combined with arbitrary words in blocks $1$ through $i-1$.  Counting the number of good words available in each block and multiplying completes the proof of the lemma.

\end{proof}

\begin{proof}[Proof of Theorem \ref{th:DC3LiY}]
Equip $\Omega$ with the tail equivalence relation
\begin{equation*}
(\omega_i)_{i=1}^\infty \thicksim (\omega'_i)_{i=1}^\infty \quad \textnormal{iff} \quad \exists n\in\mathbb{N} : \omega_n\omega_{n+1}\cdots = \omega'_n\omega'_{n+1}\cdots.
\end{equation*}
We wish to choose an uncountable set $\Lambda\subset\Omega$ such that no two points of $\Lambda$ are tail equivalent.  One quick solution is to invoke the axiom of choice and let $\Lambda$ consist of one representative point from each equivalence class of $\thicksim$.  Alternatively, we may follow a more constructive approach and take $\Lambda = \lambda\left(\{0,1\}^\mathbb{N}\right)$, where
\begin{equation*}
\lambda(\omega_1\omega_2\omega_3\cdots) = \omega_1 \ \omega_1\omega_2 \ \omega_1\omega_2\omega_3 \ \omega_1\omega_2\omega_3\omega_4 \  \cdots,
\end{equation*}
so that $\lambda(\omeg)\not\thicksim\lambda(\omeg')$ for $\omeg\neq\omeg'$.

We claim that the uncountable set $S=\{(1,0)\}\times\sigma(\Lambda) \subset M$ is pairwise DC3 scrambled for the map $F$ with respect to both metrics $\rho, \rho'$.  In what follows we will make calculations only with the metric $\rho$, because for all $s,s'\in S, n\in\mathbb{N}$ there is equality between the metrics, $\rho(F^ns,F^ns')=\rho'(F^ns,F^ns')$.

Let $s,s'\in S$ be distinct.  Such points must be of the form $s=(1,0,\sigma\omeg)$, $s'=(1,0,\sigma\omeg')$, where $\omeg, \omeg'$ are distinct points of $\Lambda$.  We have
\begin{equation*}
\begin{gathered}
F^n(s) = (\cos2\pi\sum_{j=0}^{n-1} p(\tau^j\sigma\omeg),\sin2\pi\sum_{j=0}^{n-1} p(\tau^j\sigma\omeg),\tau^n\sigma\omeg) \\
F^n(s') = (\cos2\pi\sum_{j=0}^{n-1} p(\tau^j\sigma\omeg'),\sin2\pi\sum_{j=0}^{n-1} p(\tau^j\sigma\omeg'),\tau^n\sigma\omeg').
\end{gathered}
\end{equation*}
The distance between $F^n(s)$, $F^n(s')$ is small whenever the points are on the same side of the circle, and large whenever the points are on opposite sides of the circle.  We state this notion precisely in the following implications:
\begin{multline}\label{eq:sameside}
\sum_{j=0}^{n-1} p(\tau^j\sigma\omeg) = \sum_{j=0}^{n-1} p(\tau^j\sigma\omeg') \mod 1 \implies \\
\rho(F^n(s),F^n(s')) = \max\{0,\rho_\Omega(\tau^n\sigma\omeg,\tau^n\sigma\omeg')\} \leq \textnormal{diam }\Omega \leq 1,
\end{multline}
\begin{multline}\label{eq:oppside}
\sum_{j=0}^{n-1} p(\tau^j\sigma\omeg) - \sum_{j=0}^{n-1} p(\tau^j\sigma\omeg') = \frac12 \mod 1 \implies \\
\rho(F^n(s),F^n(s')) = \max\{2,\rho_\Omega(\tau^n\sigma\omeg,\tau^n\sigma\omeg')\} = 2.
\end{multline}

Now partition the natural numbers into the sets
\begin{equation*}
\begin{aligned}
A&=\left\{ i\in\mathbb{N} ;\, \omega_1 + \cdots + \omega_i = \omega'_1 + \cdots + \omega'_i \mod 2 \right\}, \\
B&=\left\{ i\in\mathbb{N} ;\, \omega_1 + \cdots + \omega_i \neq \omega'_1 + \cdots + \omega'_i \mod 2 \right\}.
\end{aligned}
\end{equation*}
If $i\in A$ and $\omega_{i+1}\neq\omega'_{i+1}$, then $i+1\in B$.  Conversely, if $i\in B$ and $\omega_{i+1}\neq\omega'_{i+1}$, then $i+1\in A$.  Since $\omeg \not\thicksim \omeg'$, it follows that both sets $A,B$ are infinite.

To simplify the notation, we define for all $i \in \mathbb{N}$ the numbers
\begin{equation*}
\begin{gathered}
r_i=\frac12((1-\omega_1)+\cdots+(1-\omega_i)) \mod 1\\ 
r'_i = \frac12((1-\omega'_1)+\cdots+(1-\omega'_i)) \mod 1\\
\alpha_i = \frac{(2^{n_i-2}-2^{i-1}-1)\cdot \prod_{l=1}^{i-1}(2^{n_l}-2^l)}{2^{m_i-2}}.
\end{gathered}
\end{equation*}
In the expression for $\alpha_i$ if we multiply the numerator and denominator by $4$ and rearrange terms, we find
\begin{equation*}
\alpha_i = \frac{\prod_{l=1}^i(2^{n_l}-2^l)}{2^{m_i}}-\frac{2^i+4}{2^{n_i}}\cdot\frac{\prod_{l=1}^{i-1}(2^{n_l}-2^l)}{2^{m_{i-1}}}, \quad i\in\mathbb{N}
\end{equation*}
Since $n_i$ is strictly increasing in $i$, it follows that $n_i-i$ is nondecreasing.  Moreover, $n_i-i \to \infty$, for otherwise, $n_i-i$ would be eventually constant, which contradicts \eqref{eq:limalpha}.  Now we may take a limit and conclude that $\alpha_i \to \alpha$.

Lemma \ref{lem:freq} gives us the following minimum densities
\begin{equation}\label{eq:recaplemma}
\begin{gathered}
\frac{1}{2^{m_i-2}}\cdot\#\Big\{n\in\{1,\ldots,2^{m_i-2}\};\, \sum_{j=0}^{n-1}p(\tau^j\sigma\omeg)=r_i \mod 1\Big\} \geq \alpha_i \\
\frac{1}{2^{m_i-2}}\cdot\#\Big\{n\in\{1,\ldots,2^{m_i-2}\};\, \sum_{j=0}^{n-1}p(\tau^j\sigma\omeg')=r'_i \mod 1\Big\} \geq \alpha_i.
\end{gathered}
\end{equation}

Suppose that $i\in A$.  By the definition of $A$ we have $r_i=r'_i \mod 1$.  It follows from \eqref{eq:sameside} and \eqref{eq:recaplemma} that
\begin{equation*}
\frac{1}{2^{m_i-2}}\cdot\#\Big\{n\in\{1,\ldots,2^{m_i-2}\};\, \rho(F^n(s),F^n(s')) < \delta \Big\} \geq 2\alpha_i - 1, \quad \delta>1.
\end{equation*}
Taking the limit as $i\to\infty$, $i\in A$, we obtain an estimate for the upper distribution function
\begin{equation*}
\Phi_{(s,s')}^*(\delta) \geq 2\alpha - 1, \quad \delta>1.
\end{equation*}

Now suppose that $i\in B$.  By the definition of $B$ we have $r_i-r'_i = \frac12 \mod 1$.  It follows from \eqref{eq:oppside} and \eqref{eq:recaplemma} that
\begin{equation*}
\frac{1}{2^{m_i-2}}\cdot\#\Big\{n\in\{1,\ldots,2^{m_i-2}\};\, \rho(F^n(s),F^n(s')) < \delta \Big\} \leq 1 - (2\alpha_i - 1), \quad \delta<2.
\end{equation*}
Taking the limit as $i\to\infty$, $i\in B$, we obtain an estimate for the lower distribution function
\begin{equation*}
\Phi_{(s,s')}(\delta) \leq 2 - 2\alpha, \quad \delta<2.
\end{equation*}

In the hypotheses of the theorem, we fixed $\alpha>\frac34$.  It follows that 
\begin{equation*}
\Phi_{(s,s')}(\delta) \leq 2-2\alpha < \frac12 < 2\alpha-1 \leq \Phi_{(s,s')}^*(\delta), \quad 1<\delta<2.
\end{equation*}
Thus $s,s'$ are a DC3 pair.  This completes the proof.
\end{proof}


\section{Conjugacy problem}
 We show that DC3 is not preserved by conjugacy, i.e. distributional chaos of type 3 can be destroyed (or created) by using a conjugating homeomorphism. In Theorem~\ref{ThPairs}, we will assume that distributional chaos means the existence of a chaotic pair and we construct a DC3 system which is conjugated to a dynamical system without any distributionally scrambled pair. This constitutes the first correct proof of \cite[Theorem 2]{BSS}. In Theorem~\ref{Th3}, we will assume that distributional chaos means the existence of an uncountable chaotic set and we present a DC3 system which is conjugated to a dynamical system with only DC3 pairs (the maximum cardinality of any distributionally scrambled set in this system is $2$).\\

In the following lemma, $\rho$ will be the circle metric given by $\rho(\phi_a,\phi_b)= \min\{ |\phi_a-\phi_b|, 1-|\phi_a-\phi_b|\}$.
\begin{lemma}\label{LemPi}
Suppose we are given two angles $\phi_x , \phi_y \in [0,1]$ (mod 1), a natural number $k$ and a number $\delta \in (0,0.5)$. Let  $\pi_k = \pi_k(\phi_x, \phi_y, \delta) := \# \{ 0 \le i \le k-1; \rho(\phi_x + \frac{i}{k}, \phi_y)< \delta\}$. Then 
$$2\delta k- 1 \le \pi_k \le 2\delta k + 1$$
\end{lemma}

\begin{proof}
The reader may picture $k$ points equally spread around the circle $[0,1]$ (mod 1) and an arc $(\phi_y - \delta, \phi_y + \delta)$ of radius $\delta \in (0,1/2)$.  Without loss of generality we will take $\phi_x=0$. Then $\pi_k = \# \{ 0 \le i \le k-1; \rho(\frac{i}{k}, \phi_y)< \delta\} $ is equal to $ \# \{ i \in \mathbb Z; | \frac{i}{k} - \phi_y|< \delta\} = \# \{ i \in \mathbb Z; | i - \phi_y|< \delta k\} $. By these equalities we moved the problem from the circle to the real line and so the question is: How many integers lie in the open interval $(\phi_y - \delta k, \phi_y + \delta k)$ with length $ 2\delta k$? At least $2\delta k -1$ and at most $2\delta k +1$.
\end{proof}

\begin{theorem}\label{ThPairs}
The existence of DC3 pairs is not preserved by topological conjugacy.
\end{theorem}

\begin{proof}

Throughout the proof, we use the cylindrical coordinate system for $\mathbb R^3$.  This means that the point $(r\cos(2\pi\phi),r\sin(2\pi\phi),z)$, $r\ge 0$, $\phi\in[0,1]$, will be denoted more compactly as $(r, \phi, z)$.
Let $d$ be the max-metric on $\mathbb R^3$ given by $d (a,b)=\max\{|r_a-r_b|, |z_a-z_b|, \rho(\phi_a,\phi_b)\}$, where $\rho(\phi_a,\phi_b)= \min\{ |\phi_a-\phi_b|, 1-|\phi_a-\phi_b|\}$.
We construct 2 conjugate dynamical systems $(X,f)$ and $(Y,g)$.
We also denote the $j$-th iterate of a point $x$ by $f^j(x)= (r_x^j , \phi^j_x, z^j_x)$, and when no confusion can result, we use the same notation for the iterates of $x$ by $g$.

The space $X$  consists of 2 concentric columns of \lq\lq rings;\rq\rq in each column the rings are accumulating on a bottom-most ring. The definition is
$$X= \left\{ (r,\phi, z): r\in \{0.01, 0.02\}, \phi \in [0,1], z \in \{ \frac{1}{n}; n\in \mathbb N\} \cup \{ 0 \} \right\} .$$
We need the radius difference $|r_1-r_2|$ to be smaller than $1/2$. For an easier image we choose $r_1=0.01$ and $r_2=0.02$, but it is not necessary to work with these two exact numbers.

The map $f: X \to X$ carries each ring down to the next lower ring with some rotation and fixes the bottom ring. The definition is:
\begin{equation}
 f(r,\phi, \frac{1}{n}) = (r, \phi+ \varphi^{(r)}_n, \frac{1}{n+1}) \quad \mbox{and} \quad  f(r , \phi, 0) = (r , \phi, 0)
 \end{equation}
where $\phi+ \varphi^{(r)}_n$ is computed modulo 1 and the rotation angle $\varphi^{(r)}_n$ is given by
{
\renewcommand{\arraystretch}{2.1}
\begin{equation}\label{fi}
  \begin{array}{l l}
   	 \varphi^{(0.01)}_n = \dfrac{1}{k},  \mbox{ \, for } l_{k-1} < n \le l_{k} , k\in \mathbb{N}& \\
    	 \varphi^{(0.02)}_n = \left\{ \begin{array}{l  l}
  				  \dfrac{1}{k}, \mbox{ \, for } l_{k-1} < n \le l_{k},  \textnormal{\, odd } k  &\\
						\dfrac{2}{k}, \mbox{ \, for } l_{k-1} < n \le l_{k},  \textnormal{\, even } k &\\
  						\end{array} \right.
  \end{array} 
\mbox{where } \ l_k= \sum\limits_{i=1}^{k} i^{i}, \ \mbox{for} \ k \in \mathbb N \ \mbox{and} \ l_0 = 0. \\
  \end{equation}
  }

We may think of $\varphi^{(r)}_n $ as a sequence being divided into blocks, where the $k$th- block consists of the angles $\frac{1}{k}$ or $\frac{2}{k}$ repeated $k^k$ times.
 We needed to make the block length a multiple of $k$ and with the property that $\frac{l_k}{ l_{k+1}-l_k} \to 0$ for $k \to \infty $. It is easy to see that for our blocks it is true and so for $k \to \infty $
 \begin{equation}\label{Lk}
  \frac{l_k}{(k+1)^{k+1}} \to 0  \quad \mbox{ and } \quad  \frac{k^k}{l_k} \to 1 .
 \end{equation}

 We will see that the only DC3 pairs in $(X,f)$ are  points in different cylinders and which are not fixed.  The chaos is detected by distances of angles, which are always smaller than 1/2, so we will \lq\lq kill\rq\rq chaos in the conjugate system by making the distance bigger than 1/2 for every two points in different cylinders and keeping the map on each cylinder the same as before.  We can do it by lifting the inner cylinder by more than 1/2 unit, specifically in our case we choose 1.3 units and we can imagine this space like an extended telescope.
 
We construct a conjugate dynamical system $(Y,g)$ by the following definitions 
\begin{equation}\label{PI}
\begin{gathered}
\mbox{$Y=\Pi(X)$ and $g=\Pi\circ f \circ \Pi^{-1}$,}\\
\mbox{where }\Pi((r, \phi, z)) = \begin{cases}
    (r ,\phi, z +1.3), & \mbox{if $r=0.01$},\\
    (r,\phi, z),  &  \mbox{if $r=0.02$}.
\end{cases}
\end{gathered}
\end{equation}

 We will first show that in $(Y, g)$ there do not exist any DC3 pairs and after that we will show that in $(X,f)$ there are DC3 pairs. So lifting the inner column by the homeomorphism $\Pi$ destroys all DC3 pairs.

\begin{enumerate}[(A)]
\item  \label{Y} 
We claim that there are no DC3 pairs in $(Y,g)$.\\
We can see all possible different cases for pairs $(x,y) \in Y\times Y$ on the following graph, where $x~=~(r_x,\phi_x,z_x)$ and $ y~=~(r_y,\phi_y,z_y)$.\\

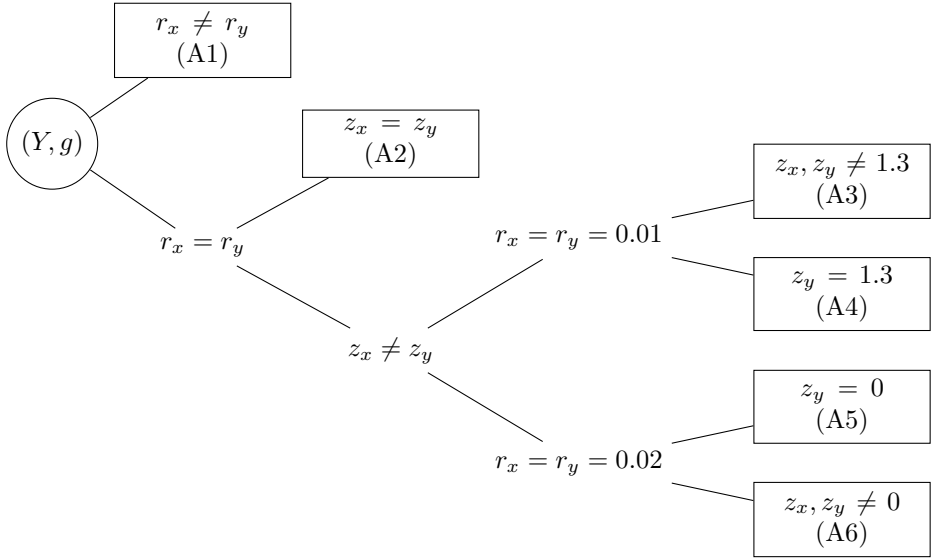
\tikzstyle{level 1}=[level distance=2cm, sibling distance=2.75cm]
\tikzstyle{level 2}=[level distance=2.5cm, sibling distance=2.75cm]
\tikzstyle{level 3}=[level distance=2.5cm, sibling distance=3cm]
\tikzstyle{level 4}=[level distance=3.5cm, sibling distance=1.5cm]
\tikzstyle{end} = [rectangle, draw, text width=6em, text centered]
\tikzstyle{1} = [circle, draw, text centered]

\begin{tikzpicture} [grow'=right, sloped]
\node[1] {$(Y,g)$} 
    child { node [end]{$r_x \neq r_y$ \\ (\ref{Y1})} }
    child { node {$r_x = r_y$}
    	child { node[end] {$z_x = z_y$ \\ (\ref{Y2})} }
        child { node {$z_x \neq z_y$}
        		child { node {$r_x=r_y=0.01$}
			child { node[end] {$z_x , z_y \neq 1.3 $  \\ (\ref{Y3})}}
			child { node[end] {$ z_y = 1.3 $ \\ (\ref{Y4}) }}
		}	
		child  { node {$r_x=r_y=0.02$}
           	 	child {node[end] {$ z_y = 0 $ \\ (\ref{Y5}) }}
			child {  node[end] {$z_x , z_y \neq 0 $ \\ (\ref{Y6}) }}  
		}
	}		
	};
\end{tikzpicture}

\begin{enumerate}[(\ref{Y}1)]
\item  \label{Y1} $r_x \neq r_y$\\
Since $|r^j_x - r^j_y|= 0.01$, $\rho(\phi^j_x, \phi^j_y) \in [0,1/2] $  for all $j$ and by (\ref{PI}) $\displaystyle\lim_{j\to\infty}|z^j_x - z^j_y| = 1.3$,  by the maximal metric  we get:
$
\displaystyle\lim_{j\to\infty} d(g^j(x), g^j(y))=1.3
$.
 That means any such $(x,y)$ is not DC3.\\

\item  \label{Y2} $r_x = r_y=r$, $z_x = z_y=z$  ($\phi_x \neq \phi_y$)\\
If $z=0$ or $z=1.3$, then $x$ and $y$ are fixed points and we are done. \\So let's say $z= \frac{1}{c}$ or $z=\frac{1}{c} +1.3$, where $c\in \mathbb N$. 
Since  $z_x = z_y$ and $r_x = r_y$

\begin{minipage}{\dimexpr\textwidth-2\leftmargin\relax}
$$
d(g^j(x), g^j(y))= \rho(\phi^j_x, \phi^j_y) = \rho\left[\left(\phi_x + \sum_{n=c}^{c+j-1}\varphi^{(r)}_n \right) ,\left( \phi_y + \sum_{n=c}^{c+j-1}\varphi^{(r)}_n\right)\right] = \rho(\phi_x, \phi_y) .
$$
\end{minipage}

That means $\displaystyle\lim_{j\to\infty} d(g^j(x), g^j(y))=\rho(\phi_x, \phi_y)$ and so $(x,y)$ is not DC3.\\

\item  \label{Y3} $z_x \neq z_y$ and $z_x,z_y \neq1.3$, $r_x = r_y= r = 0.01$\\
Let $z_x = \frac{1}{c_x}+1.3, \, z_y = \frac{1}{c_y}+1.3$, and without loss of generality we can say $c_x < c_y$.
Then $\displaystyle\lim_{j\to\infty}|z^j_x - z^j_y|=\displaystyle\lim_{j\to\infty}\left |\left(\tfrac{1}{c_x +j} + 1.3\right) - \left(\tfrac{1}{c_y + j} + 1.3\right) \right|= 0$ and $|r^j_x - r^j_y|= 0$ for any $j\in \mathbb N$. Because of this, it is enough to check $\rho(\phi^j_x, \phi^j_y)$,

\begin{minipage}{\dimexpr\textwidth-2\leftmargin\relax}
$$\displaystyle\lim_{j\to\infty} d (g^j(x), g^j(y))= \displaystyle\lim_{j\to\infty}\rho(\phi^j_x, \phi^j_y)= 
\displaystyle\lim_{j\to\infty} \rho\left[\left(\phi_x +  \sum_{n=c_x}^{c_x+j-1}\varphi^{(r)}_n\right) , \left(\phi_y + \sum_{n=c_y}^{c_y+j-1}\varphi^{(r)}_n \right)\right] .$$
\end{minipage}

Writing $c_y = c_x + c$ and expanding the sums, this becomes

\begin{minipage}{\dimexpr\textwidth-2\leftmargin\relax}
\begin{equation*}
\lim_{j\to\infty} \rho\left[\left(\phi_x +  \sum_{n=c_x}^{c_x+c -1}\varphi^{(r)}_n +  \sum_{n=c_x + c}^{c_x+j-1}\varphi^{(r)}_n\right) , \left( \sum_{n=c_x + c}^{c_x+j-1}\varphi^{(r)}_n + \sum_{n=c_x + j }^{c_x + j +c-1}\varphi^{(r)}_n + \phi_y  \right)\right].
\end{equation*}
\end{minipage}

We may cancel the innermost sums.  By (\ref{fi}) we find $\lim_{j\to\infty} \sum_{n=c_x+j}^{c_x+j+c-1}\varphi^{(r)}_n = 0$.  Therefore
$\displaystyle\lim_{j\to\infty}\rho(\phi^j_x, \phi^j_y) $ exists.
That means the limit $\displaystyle\lim_{j\to\infty} d(g^j(x), g^j(y))$ exists and so $(x,y)$ is not DC3.\\ 

\item  \label{Y4} $z_x \neq z_y$ and $z_y =1.3$, $r_x = r_y=0.01$\\
As we already mentioned for these cases
$\displaystyle\lim_{j\to\infty}|z^j_x - z^j_y|= 0$ and $|r^j_x - r^j_y|= 0$  for all $j$.
That means for every $\delta$ there exists $m_\delta$ such that for all $j > m_\delta$: $|z_x^j - z_y^j| < \delta$  and so 

\begin{equation} 
\begin{gathered}
\#\{0 \le j \le n;d(g^j(x),g^{j}(y))<\delta\}=\#\{0 \le j \le n;\rho(\phi^j_x, \phi^j_y)<\delta\} - c_\delta \\
\mbox{where} \quad 
c_\delta= \#\{0\le j \le m_\delta ;|z_x^j - z_y^j| \ge \delta\ \mbox{ and \, } \rho(\phi^j_x, \phi^j_y)<\delta\} .
\end{gathered}
\end{equation}
We can see that $c_\delta$ is just some final constant (for each $\delta$) which for $n~\to~\infty $ does not matter and so we can calculate  $\Phi^*_{(x, y)}(\delta)$ and $ \Phi_{(x, y)}(\delta)$ using $\rho(\phi^j_x, \phi^j_y)$ in place of $d(g^j(x), g^j(y))$.
Moreover, notice that $ \rho(\phi^j_x, \phi^j_y) \in [0,0.5]$ for all $j$, so  $ \Phi^*_{(x, y)}(\delta)= \Phi_{(x, y)}(\delta)=1$ for $\delta > 0.5$.
From the definition of the map $g$ our $y= (0.01, \phi_y, 1.3)$ is a fixed point, and for simplicity let's take $z_x=1+1.3$ (if it is not, we replace $x$ by one of its preimages). Then we can write\\
\begin{minipage}{\dimexpr\textwidth-2\leftmargin\relax}
$$
   \rho(\phi^j_x, \phi^j_y)~=~\rho\left(\phi_x~+~\displaystyle\sum_{n=1}^{j}\varphi^{(r)}_n~,~\phi_y\right).
$$
\end{minipage}\\
We may think about $\phi^j_x$ in blocks. In the $k$th-block (between $l_{k-1}$ and $ l_k$) we can see the sequence of these $k$ angles
$\{ \phi^{l_{k-1}}_x + 0, \phi^{l_{k-1}}_x + \frac{1}{k}, \phi^{l_{k-1}}_x + \frac{2}{k},...,\phi^{l_{k-1}}_x + \frac{k-1}{k} \}$ exactly $k^{k-1}$ times.
Then by Lemma~\ref{LemPi} we obtain
\begin{equation} \label{pik}
\begin{gathered}
2\delta k - 1 \le \pi_k  \le 2\delta k + 1 \\
\mbox{where} \quad 
\pi_k := \#\{l_{k-1}\le j \le l_{k-1} + k-1; \rho(\phi_x^j, \phi_y) < \delta \}
\end{gathered}
\end{equation}
and so in the whole  $k$-th block: 
\begin{equation} \label{PIK}
\begin{gathered}
k^k 2\delta - k^{k-1} \leq B_k  \leq 2\delta k^k +  k^{k-1} \\
\mbox{where} \quad 
B_k := \#\{l_{k-1}\le j < l_{k}; \rho(\phi_x^j, \phi_y) < \delta \} = k^{k-1} \pi_k \ .
\end{gathered}
\end{equation}

To show equality of $\Phi^*_{(x, y)}(\delta)=\Phi_{(x, y)}(\delta)$ we need to consider all natural numbers $n$.
For each $n$ there exist unique $k_n,\alpha_n$ and $\beta_n$ such that $\alpha_n < k_n^{(k_n-1)}$, $\beta_n < k_n$ and
\begin{equation} \label{N}
	n= 1^1 + 2^2 +...+ (k_n-1)^{(k_n-1)} + \alpha_n k_n + \beta_n = l_{k_n-1} + \alpha_n k_n + \beta_n .
	\end{equation}

Let's mark $ p_n := \#\{0 \le j \le n;\rho(\phi^j_x, \phi^j_y)<\delta\}$. Then by (\ref{PIK}) and (\ref{N}), for each $n$ there exists a unique $\gamma_n~\le~\pi_{k_n}$ such that 
\begin{equation} \label{P}
\begin{gathered}
p_n =  \sum\limits_{j=1}^{k_n-1}B_{j}  + \alpha_n \pi_{k_n} + \gamma_n =
		1^0\pi_1 + ...+ (k_n-1)^{(k_n-2)}\pi_{k_n-1} + \alpha_n \pi_{k_n} + \gamma_n  =\\
 = \sum\limits_{j=1}^{k_n-1} j^{j-1}\pi_{j}  + \alpha_n \pi_{k_n} + \gamma_n .
\end{gathered}
\end{equation}	

	By (\ref{pik}) and (\ref{P})
	$$\qquad \qquad \quad p_n \le
	\sum\limits_{j=1}^{k_n-1} (j^{j-1}(2\delta j +1))  + \alpha_n (2\delta k_n +1) + \gamma_n 
	= \sum\limits_{j=1}^{k_n-1} (j^{j}(2\delta)) + \sum\limits_{j=1}^{k_n-1} j^{j-1} +\alpha_n (2\delta k_n +1)+ \gamma_n .$$
	Hence
	\begin{equation} \label{pocet1}
p_n \le l_{k_n-1} 2 \delta + \sum\limits_{j=1}^{k_n-1} j^{j-1} +\alpha_n (2\delta k_n +1)+ \gamma_n .
	\end{equation}

And similarly
\begin{equation} \label{pocet2}
	p_n \ge
		l_{k_n-1} 2 \delta  -  \sum\limits_{j=1}^{k_n-1} j^{j-1}  + \alpha_n (2\delta k_n - 1) + \gamma_n .
	\end{equation}

By (\ref{Lk}), (\ref{N}) and (\ref{pocet1}) we get
\begin{equation} \label{Up1}
\begin{gathered}
\Phi^*_{(x, y)}(\delta) \le 	
\displaystyle\limsup_{n\to\infty}\dfrac{l_{k_n-1} 2 \delta + \sum\limits_{j=1}^{k_n-1} j^{j-1} +\alpha_n (2\delta k_n +1)+ \gamma_n}{l_{k_n-1} + \alpha_n k_n + \beta_n} = \\
= \displaystyle\limsup_{n\to\infty}\left[\dfrac{ 2 \delta(l_{k_n-1}    + \alpha_n k_n) + \gamma_n }{l_{k_n-1} + \alpha_n k_n + \beta_n } 
+
\dfrac{\sum\limits_{j=1}^{k_n-1} j^{j-1} + \alpha_n  }{l_{k_n-1} + \alpha_n k_n + \beta_n} \right] = 2\delta .\\
\end{gathered}
\end{equation}

and by (\ref{Lk}), (\ref{N}) and (\ref{pocet2}) we get
\begin{equation} \label{Lo1}
\Phi_{(x, y)}(\delta) \ge 	
\displaystyle\liminf_{n\to\infty}\left[\dfrac{ 2 \delta (l_{k_n-1}   + \alpha_n  k_n) + \gamma_n }{l_{k_n-1} + \alpha_n k_n + \beta_n } 
-
\dfrac{  \sum\limits_{j=1}^{k_n-1} j^{j-1} +\alpha_n  }{l_{k_n-1} + \alpha_n k_n + \beta_n} \right] = 2\delta .
\end{equation}	

From (\ref{Up1}) and (\ref{Lo1}), we obtain
\begin{equation} 
\begin{gathered}
\Phi_{(x, y)}(\delta) = \Phi^*_{(x, y)}(\delta) = 2\delta, \quad \mbox{for } \delta < 1/2 \\
\Phi_{(x, y)}(\delta) = \Phi^*_{(x, y)}(\delta) = 1, \quad \mbox{for } \delta \ge 1/2 .
\end{gathered}
\end{equation}
This shows that $(x,y)$ is not DC3.\\

\item  \label{Y5} $z_x \neq z_y$ and $z_y=0$, $r_x = r_y=0.02$\\
This case is quite similar to (\ref{Y4}), so for the same reasons we will use $\rho(\phi^j_x, \phi^j_y)$ in place of  $d(g^j(x), g^j(y))$ for computing the upper and lower distributional function.
We can also see, that $y$ is again a fixed point and for the same reason as in (\ref{Y4}) we can take $z_x=1$.  We will also use the same notation as in (\ref{Y4}).
We may again think in blocks indexed by $k$, but now we have to distinguish between even and odd $k$. 
For $k$-odd, blocks are exactly the same as in (\ref{Y4}) - see (\ref{pik} - \ref{P}). 
But for even $k$ in every iteration we add the angle $2/k$ instead of $1/k$, then by using Lemma~\ref{LemPi} we get
\begin{equation}
\begin{gathered}
 \mbox{ $k$ is odd} \implies 2\delta k - 1 \le \pi_k  \le 2\delta k + 1 \\
 \mbox{ $k$ is even} \implies 2\delta\frac{k}{2} - 1 \le \pi_k  \le 2\delta \frac{k}{2} + 1 \\
 \mbox{ where } \pi_k :=   \begin{cases}
  				   \#\{l_{k-1}\le j \le l_{k-1} + k-1; \rho(\phi_x^j, \phi_y) < \delta \}, &\mbox{  for odd }  k\\
						 \#\{l_{k-1}\le j \le l_{k-1} + \frac{k}{2} -1; \rho(\phi_x^j, \phi_y) < \delta \},  &\mbox{ for even } k . \\
  						\end{cases}
 \end{gathered}
  \end{equation}

You can see that for $k$-even we get full rotation in $\frac{k}{2}$ steps instead of $k$ steps, so for the whole block we get

 \begin{equation} 
 B_k =   \begin{cases}
  k^{k-1}\pi_k ,& \mbox{for odd $k$} \\
 2k^{k-1}\pi_k, & \mbox{for even $k$}. \end{cases}
 \end{equation}
 
 Hence
\begin{equation} 
k^k 2\delta-2k^{k-1} \le B_k \le k^k 2\delta +2k^{k-1}, \quad \mbox{for every $k \in \mathbb N$.} 
\end{equation}
Then by same computation as in (\ref{Y4}) we obtain
\begin{equation} 
\begin{gathered}
\Phi_{(x, y)}(\delta) = \Phi^*_{(x, y)}(\delta) = 2\delta, \quad \mbox{for } \delta < 1/2 \\
\Phi_{(x, y)}(\delta) = \Phi^*_{(x, y)}(\delta) = 1, \quad \mbox{for } \delta \ge 1/2 .
\end{gathered}
\end{equation}
This shows that $(x,y)$ is not DC3.\\

\item  \label{Y6} $z_x \neq z_y$ and $z_x,z_y \neq 0$,  $r_x = r_y= r = 0.02$\\
The computing is the same as in (\ref{Y3})
so $\displaystyle\lim_{j\to\infty} d(g^j(x), g^j(y))$ converges and $(x,y)$ is not DC3.\\

\end{enumerate}

\item  \label{X} 
We claim that there is  a DC3 pair in $(X,f)$. \\
We will show that 
$x= (0.01, 0, 1)$ and $y= (0.02, 0, 1)$ is DC3.
Then $|r_x-r_y|= |r^j_x - r^j_y|= 0.01$ and $ |z^j_x - z^j_y|=0$ for all $j$ and $\rho(\phi^j_x, \phi^j_y) \in [0,0.5] $. 
We will compute $\Phi^*_{(x, y)}(\delta)$ and $\Phi_{(x, y)}(\delta)$ for $\delta \in [0.01, 0.5]$ and from the  previous sentence it is clear that for computation we may use just  $\rho(\phi^j_x, \phi^j_y)$.\\
After every $l_k$ steps (in the end of every $k$th block), $\phi^{l_k}_x = \phi^{l_k}_y=0$. Let's denote
\begin{equation} \label{Ak}
 A_k =  \# \{0 \le j \le l_k  ;d(f^j(x),f^{j}(y))<\delta\} \le l_k +1 .
 \end{equation}
After $l_{k-1}$ steps,  if $k-1$ is even, then for the next $k^{k}$ steps $\phi^j_x = \phi^j_y$ so $d(f^j(x), f^j(y))=0.01$ and
\begin{equation} \label{U}
 \# \{l_{k-1} < j \le l_{k-1} + k^{k} = l_k ;d(f^j(x),f^{j}(y))<\delta\} = k^{k} .
 \end{equation} 
 Then by (\ref{Lk}), (\ref{Ak}) and (\ref{U})  we get\\
 \begin{equation} \label{Up}
  \Phi^*_{(x, y)}(\delta) \ge \displaystyle\limsup_{k\to\infty}\dfrac{A_{k-1} +  k^{k}}{l_{k-1} +  k^{k}}= \displaystyle\limsup_{k\to\infty}\dfrac{A_{k-1} +  k^{k}}{l_{k}} = 1.
   \end{equation} \\
If $k-1$ is odd, then again we know that $ \phi^{l_{k-1}}_x = \phi^{l_{k-1}}_y$ and for the next $k^k$ steps, the inner point is rotating with speed $1/k$ while the outer point is rotating with speed $2/k$. From the point of view of calculating distances, the situation is exactly the same as if the inner point made no rotation and the outer point rotated with speed $1/k$. Applying Lemma~\ref{LemPi} (the same as we did in (\ref{PIK}))  we obtain
\begin{equation} \label{L}
 \# \{l_{k-1} < j \le l_{k} ;d(f^j(x),f^{j}(y))<\delta\} \le \# \{l_{k-1} < j \le l_{k} ;\rho (\phi^j_x,\phi^j_y )<\delta\} \le  2\delta k^{k} + k^{k-1} .
 \end{equation}\\
and by (\ref{Lk}), (\ref{Ak}) and (\ref{L}) we get 
 \begin{equation} \label{Lo}
  \Phi_{(x, y)}(\delta)\le \displaystyle\liminf_{k\to\infty}\dfrac{A_{k-1} +  2\delta k^{k } + k^{k-1}}{l_k}  = 2\delta .
   \end{equation} 
By (\ref{Up}) and (\ref{Lo}) we get
\begin{equation} 
\begin{gathered}
\Phi_{(x, y)}(\delta) = \Phi^*_{(x, y)}(\delta) = 0, \quad \mbox{for } \delta < 0.01 \\
\Phi_{(x, y)} \le 2\delta < 1 \le \Phi^*_{(x, y)}(\delta), \quad \mbox{for } \delta \in (0.01,0.5) \\
\Phi_{(x, y)}(\delta) = \Phi^*_{(x, y)}(\delta) = 1, \quad \mbox{for } \delta \ge 1/2 .
\end{gathered}
\end{equation}\\
 That shows that $(x,y)$ is a DC3 pair. 
\end{enumerate}
 \end{proof}
 

\begin{theorem}\label{Th3}
Distributional chaos of type 3 (assuming the existence of an uncountable scrambled set) is not preserved by conjugacy.
\end{theorem}
\begin{proof}
Let $I$ be the unit interval and $C$ be the middle-thirds Cantor set inside $I$. Let $\tilde{C}$ be $C$ translated by 1, i.e. $\tilde{C} = C + 1$ and $\mathbb{S}_r$ be the circle with radius $r\in\tilde{C}$ and with center at the origin. The desired counterexample is constructed on the union in $\mathbb{R}^3$ of uncountably many concentric cylinders above $\{\mathbb{S}_r:r\in\tilde{C}\}$ (height of points in these cylinders ranges only in $\{\frac{1}{k}:k\in\mathbb{N}\}\cup\{0\}$), with a skew-product mapping. The base map of this skew-product mapping decreases the height of a point in the limit to 0 and each fiber map is a rotation by an angle $p(k)$, where $p:\mathbb{N}\rightarrow I$ is a continuous mapping.

 Define $$X=\left\{(r\cos(2\pi\varphi),r\sin(2\pi\varphi),z): r\in \tilde{C},\varphi\in I, z\in \{\frac{1}{k}:k\in\mathbb{N}\}\cup\{0\}\right\}.$$
The transformation $f:X\rightarrow X$ is the identity for points with zero third coordinate,
$$(r\cos(2\pi\varphi),r\sin(2\pi\varphi),0)\mapsto (r\cos(2\pi\varphi,r\sin(2\pi\varphi),0),$$
and for the other points $f$ is defined as follows
$$\left(r\cos(2\pi\varphi),r\sin(2\pi\varphi),\frac{1}{k}\right ) \mapsto \left(r\cos(2\pi(\varphi+p(k))),r\sin(2\pi(\varphi+p(k))),\frac{1}{k+1}\right ).$$
To define the function $p:\mathbb{N}\rightarrow I$, let $0<n_0<n_1<\ldots$ be an increasing sequence of integers which will be specified later and 
\begin{equation}\label{gr}p(k)= \left\{
  \begin{array}{l l l}
    \frac{1}{2^{m+1}} &  \text{ if } s_m<k\leq s_m+2^m\\
    0& \text { if } s_m+2^m<k\leq s_m+n_m,
  \end{array} \right.
  \end{equation}
  where $s_m=\sum_{i=0}^{m-1}n_i$ and $s_0=0$.
  Let $(Y,F)$ be a dynamical system conjugated with $(X,f)$ via the following homeomorphism
\begin{equation}\label{pr}\Pi((x,y,z)) = \left\{
  \begin{array}{l  l}
    (x,y,z) & \mbox{if $x\geq 0$}\\
    \left(2x,y,z\right) &  \mbox{if $x< 0$}\\
  \end{array} \right.
  \end{equation}
  where $Y=\Pi(X)$ and $F=\Pi\circ f \circ \Pi^{-1}$.
 
  I. $(Y,F)$ is DC3\\
  We claim that the set $S\subset Y$,
  $$S=\{(r,0,1): r\in \tilde{C}\},$$
  is distributionally scrambled of type 3. 
 Denote $f^i(x)=\left(r\cos(2\pi\eta^i),r\sin(2\pi\eta^i),z^i\right)$, for $i\geq 0$. Then for $x\in S$ $$z^i=\frac{1}{1+i}$$ and $$\eta^i=\Sigma_{j=1}^i p(j)\mod 1.$$
  Let
$$U_k=\#\{0\leq i\leq s_k+n_k; \eta^i=0\}, \quad L_k=\#\{0\leq i\leq s_k+n_k; \eta^i=\frac{1}{2}\}.$$
By (\ref{gr}) and the definition of $\eta^i$, we can see that for $k\geq 0$
$$L_{2k}\geq n_{2k}-2^{2k}\quad\text{and}\quad U_{2k+1}\geq n_{2k+1}-2^{2k+1}.$$
Let $x_1, x_2$ be two distinct points in $S$ such that $x_1=(r_1,0,1)$ and $x_2=(r_2,0,1)$.  Then $d(x_1,x_2)=|r_1-r_2|=\epsilon.$ By the definition of $F$,
  $$d(F^i(x_1),F^i(x_2))=d(f^i(x_1),f^i(x_2))=\epsilon, \text{ if } \eta_i=0$$
  and
  $$d(F^i(x_1),F^i(x_2))=2\cdot d(f^i(x_1),f^i(x_2))=2\epsilon,\text{ if } \eta_i=\frac 12.$$
  It follows for any $\delta\in[\epsilon,2\epsilon],$
  \begin{equation}\label{R*}\Phi^*_{(x_1,x_2)}(\delta)\geq \lim_{k\to\infty} \frac{1}{s_{2k+1}+n_{2k+1}}U_{2k+1}\geq \lim_{k\to\infty}\frac{n_{2k+1}-2^{2k+1}}{n_{2k+1}+s_{2k+1}}
  ,\end{equation}
  \begin{equation}\label{R}\Phi_{(x_1,x_2)}(\delta)\leq 1-\lim_{k\to\infty} \frac{1}{s_{2k}+n_{2k}}L_{2k}\leq 1-\lim_{k\to\infty}\frac{n_{2k}-2^{2k}}{s_{2k}+n_{2k}}.\end{equation}
We set the sequence $0<n_0<n_1<\ldots$ in such a way that 
$$\lim_{k\to\infty}\frac{s_k}{n_k}=0\quad\text { and simultaneously }\quad\lim_{k\to\infty}\frac{2^k}{n_k}=0.$$
By (\ref{R*}) and (\ref{R}), we get for any $\delta\in [\epsilon,2\epsilon]$,
  $$\Phi^*_{(x_1,x_2)}(\delta)=1, \quad \Phi_{(x_1,x_2)}(\delta)=0.$$
  
 II. $(X,f)$ is not DC3 \\
 It is sufficient to show that every DC3 pair in $X$ contains a fixed point and therefore the maximum cardinality of any distributionally scrambled set in $X$ is $2$. Let $$Fix=\{(r\cos 2\pi\varphi,r\sin 2\pi\varphi,0): r\in\tilde{C},\varphi\in I\}$$
 be the set of all fixed points. We prove that for every $(x,y)\in(X \setminus Fix)^2$ the limit $\lim_{n\to\infty}d(f^n(x),f^n(y))$ always exists and hence $\Phi^*_{(x,y)}(\delta)=\Phi_{(x,y)}(\delta),$ for any $\delta>0$. We assume the max-metric $d(x,y)=\max \{|r_x-r_y|,\rho(\varphi_x,\varphi_y),|\frac{1}{k_x}-\frac{1}{k_y}|\}$ in the rest of the proof, where $\rho$ is the metric on the unit circle $\rho(\alpha,\beta)=\min \{|\alpha-\beta|,1-|\alpha-\beta| \}$.

Let $x=\left(r_x\cos 2\pi\varphi_x, r_x\sin 2\pi\varphi_x,\frac{1}{k_x}\right)$ and $y=\left(r_y\cos 2\pi\varphi_y, r_y\sin 2\pi\varphi_y,\frac{1}{k_y}\right)$. We assume without loss of generality that $k_x\leq k_y$, and after replacing $x$ and $y$ by their $(k_x - 1)$th preimages, we may assume that $k_x=1$ and we may write $l=k_y$.  By the definition of $f$, the distance $|r_{f^n(x)} - r_{f^n(y)}|=|r_x - r_y|$ is constant and the distance $|\frac{1}{k_{f^n(x)}}-\frac{1}{k_{f^n(y)}}|$ decreases to zero.  Moreover, using the fact that $\lim_{n\to\infty} \sum_{k=n+1}^{n+l} p(k)=0,$ we find that the difference between the angles (with all calculations modulo 1) converges to a constant:
 \begin{multline*}
\lim_{n\to\infty} \varphi_{f^n(x)}-\varphi_{f^n(y)} = \lim_{n\to\infty} \left(\varphi_x+\sum_{k=1}^n p(k)\right) - \left(\varphi_y+\sum_{k=l}^{n+l} p(k)\right) = \\
 = \varphi_x - \varphi_y + \sum_{k=1}^{l-1} p(k) - \lim_{n\to\infty} \sum_{k=n+1}^{n+l} p(k) = \varphi_x - \varphi_y + \sum_{k=1}^{l-1} p(k).
 \end{multline*}
 It follows that $\rho(\varphi_{f^n(x)},\varphi_{f^n(y)})$ converges to a constant.  The existence of the limit $\lim_{n\to\infty} d(f^n(x),f^n(y))$ follows from the separate convergence in all three coordinates.
\end{proof}

\section{Distributional chaos of type $2\frac 12$}
Nevetherless, the notion of DC3 can be strengthened in such a way that it is preserved under conjugacy and implies Li-Yorke chaos:\\

\begin{definition}
A pair $(x_1, x_2)\in X^2$ is called 
\emph{distributionally scrambled of type $2\frac 12$} if there are positive numbers $c$ and $s$ such that, for any $0<\delta<s$,
$$\Phi_{(x_1, x_2)}(\delta)<c< \Phi^*_{(x_1, x_2)}(\delta).$$
\end{definition}

By simple observation we can see that if $(x_1, x_2)\in X^2$ is $DC2\frac{1}{2}$, then it is Li-Yorke scrambled. Since $ \Phi^*_{(x_1, x_2)}(\delta)>c$ for arbitrary small $\delta$, $(x_1, x_2)$ must be proximal (for distal pairs $\Phi^*_{(x_1, x_2)}(\epsilon)=0$ for some $\epsilon>0$). Similarly, $(x_1, x_2)$ is not asymptotic (for asymptotic pairs $\Phi_{(x_1, x_2)}(\delta)=1$ for every $\delta>0$).\\

\begin{remark}
 We call a dynamical system \emph{strictly $DC2\frac{1}{2}$} if it possess an uncountable $DC2\frac{1}{2}$ set but no $DC2$ pairs. By results in \cite{DL}, positive topological entropy implies existence of an uncountable DC2 set, hence strictly $DC2\frac{1}{2}$ systems must have zero topological entropy.
\end{remark}

\begin{theorem}\label{4}
 Let $f$ and $g$ be topologically conjugate continuous maps of a compact metric space $(X,d)$. Then $f$ is $DC2\frac{1}{2}$ if and only if $g$ is $DC2\frac{1}{2}$.
 \end{theorem}

\begin{proof}
Let $h$ be a homeomorphism conjugating $f$ and $g$ such that $f=h\circ g \circ h^{-1}$. Let $\varepsilon >0$. Since $h$ is a homeomorphism, there is a $\delta >0$ such that for any  $u,v\in X$,
\begin{equation}\label{N1}d(u,v)<\delta \mbox{  implies  } d(h(u),h(v))<\epsilon,\end{equation}
\begin{equation}\label{N2}d(h(u),h(v))<\delta \mbox{  implies  } d(u,v)<\epsilon.\end{equation}
Since $h\circ f^n=g^n\circ h$, it follows by (\ref{N1}) $$d(f^n(u),f^n(v))<\delta \mbox{  implies  } d(g^n\circ h(u),g^n \circ h(v))<\epsilon,$$
and consequently $\Phi^*_{(u,v)}(\delta)\leq \Psi^*_{(h(u),h(v))}(\epsilon)$, where $\Phi^*$ and $\Psi^*$ are the upper distribution functions of $f$ and $g$ respectively. Similarly by (\ref{N2}), $\Psi_{(h(u),h(v))}(\delta)\leq\Phi_{(u,v)}(\epsilon)$.  Then $c<\Phi^*_{(u,v)}(\delta)\leq \Psi^*_{(h(u),h(v))}(\epsilon)$, for any $0<\epsilon<s$, and there is a $0<\delta<s$ such that $\Psi_{(h(u),h(v))}(\delta)\leq\Phi_{(u,v)}(\epsilon)<c$. We can choose $\delta$ arbitrary small, hence we claim that
for any $0<\eta<\delta$,
$$\Psi_{(h(u),h(v))}(\eta)<c< \Psi^*_{(h(u),h(v))}(\eta).$$
\end{proof}
The following example illustrates that $DC2\frac{1}{2}$ is essentially weaker than $DC2$. \\
\paragraph{\bf Example}Let $X$ be the union of a converging sequence of unit fibers and the limit fiber,
  $$X=\left\{\frac{1}{k}:k\in\mathbb{N}\right\}\times I\cup \{0\}\times I,$$
  and $f$ be a skew-product map $f:X\rightarrow X$ which is the identity on the limit fiber,
  $$(0,z) \mapsto (0,z),$$
  and which is $f_k$ on inner fibers,
  $$\left(\frac{1}{k},z\right) \mapsto \left(\frac{1}{k+1},f_k(z)\right),\qquad k\in\mathbb{N}.$$
  To  define $f_k:I\rightarrow I$, let  $0<m_1<m_2<m_3<\ldots$ be an increasing sequence of integers satisfying
  $$\lim_{l\to\infty}\frac{m_l}{m_{l+1}-m_l}=0, \quad \lim_{l\to\infty}\frac{l}{m_{l}}=0,$$ and the difference $(m_{l}-m_{l-1})$ is divisible by $7l$, for any $l\in\mathbb{N}$. Let $h_l(x)=l^{-1/l}x$ and $\bar{h_l}(x)=\min \{1,l^{1/l}x \}$. The sequence $\{f_k\}^{\infty}_{k=1}$ is defined in two ways - if $l$ is odd, then $f_k$ for $m_l<k<m_{l+1}$ is formed from repeated blocks $(h_l)^l ,(Id)^{4l}, ( \bar{h_l})^l, (Id)^l$, where $(h)^l$ means $\underbrace{h,\ldots,h}_{l-\mbox{times}}$.\\
   If $l$ is even, then we change the order and $f_k$ for $m_l<k<m_{l+1}$ is formed from blocks $(h_l)^l ,(Id)^{l}, ( \bar{h_l})^l, (Id)^{4l}.$
  Let
\begin{equation}\label{definition_f}f_k = \left\{
  \begin{array}{l l l}
    h_l&\quad \text {if $m_{l}+i7l\leq k< m_{l}+i7l+l$}&\\
    Id & \quad \text {if $m_{l}+i7l+l\leq k< m_{l}+i7l+5l$}&\qquad k\in\mathbb{N},l\in2\mathbb{N}+1,n\in2\mathbb{N}\\
    \bar{h_l}& \quad \text{if $m_{l}+i7l+5l\leq k< m_{l}+i7l+6l$}&\qquad i\in\{0,1,\ldots,\frac{m_{l+1}-m_{l}}{7l}-1\},\\
    Id & \quad \text{if $m_{l}+i7l+6l\leq k< m_{l}+i7l+7l$}&\qquad  j\in\{0,1,\ldots,\frac{m_{n+1}-m_{n}}{7n}-1\},\\
  
   h_n&\quad \text {if $m_{n}+j7n\leq k< m_{n}+j7n+n$}&\\
    Id & \quad \text {if $m_{n}+j7n+n\leq k< m_{n}+j7n+2n$}&\\
    \bar{h_n}& \quad \text{if $m_{n}+j7n+2n\leq k< m_{n}+j7n+3n$}&\\
    Id & \quad \text{if $m_{n}+j7n+3n\leq k< m_{n}+j7n+7n$}&\\
    
  \end{array} \right.
  \end{equation}
  where $2\mathbb{N}=\{2n:n\in\mathbb{N}\}$ and $2\mathbb{N}+1=\{2n+1:n\in\mathbb{N}\}$.
 
The sequence $\{f_k\}^{\infty}_{k=1}$ uniformly converges to the identity and therefore $f$ is continuous.  Let $(u,v)\in\{1\}\times I$ and $\alpha=d(u,v)$. After $l$ applications of $h_l$ the distance is contracted from $\alpha$ to $\frac{\alpha}{l}$ and it remains  the same until $l$ applications of $\bar{h_l}$ recover the distance to the original $\alpha$. The identity mappings between $h_l$ and $\bar{h_l}$ ensure that the distance is contracted to $\frac{\alpha}{l}$ for four sevenths of the times $i$ and then it is recovered to $\alpha$ for one seventh of the times $i$, when $m_{l}\leq i<m_{l+1}$ and $l$ is odd.
Hence $\frac{6}{7}\geq \Phi^*_{(u,v)}(\delta)\geq \frac{4}{7}$, for any $0<\delta<\alpha$. When $l$ is even, the distance is contracted to $\frac{\alpha}{l}$ for one seventh of the times $i$ and then it is recovered to $\alpha$ for four sevenths of the times $i$, for $m_{l}\leq i<m_{l+1}$. Therefore $\frac{3}{7}\geq \Phi_{(u,v)}(\delta)\geq \frac{1}{7}$, for any $0<\delta<\alpha$ and the set $\{1\}\times I$ is $DC2\frac{1}{2}$ but not $DC2$.\\
With the same reasoning we conclude that pairs in any fiber $\{\frac{1}{k}\}\times I$ are either $DC2\frac{1}{2}$ but not $DC2$ scrambled or they have eventually equal trajectories (because of the constant part of function $\bar{h_l}$). If $u\in\{\frac 1 k\}\times I$ and $v\in\{\frac {1}{k+n}\}\times I$, then we apply functions to $u$ with a delay of $n$ times. Blocks of identities are arbitrary long and $n$ is fixed, hence the delay does not change the limits - $\frac{6}{7}\geq \Phi^*_{(u,v)}(\delta)\geq \frac{4}{7}$ and $\frac{3}{7}\geq \Phi_{(u,v)}(\delta)\geq \frac{1}{7}$, for any $0<\delta<\beta$, where $\beta=d(f^n(u),v)$. In this case $(u,v)$ is also a $DC2\frac{1}{2}$ but not $DC2$ pair. If $\beta=0$, then $v\in Orb^+_f(u)$ and $(u,v)$ is asymptotic.\\
Notice that points in $\{0\}\times I$ are fixed and hence there are no scrambled pairs in $\{0\}\times I$. Suppose $u\not\in\{0\}\times I$ and $v\in\{0\}\times I$. Then $\frac{6}{7}\geq \Phi^*_{(u,v)}(\delta)$, for sufficiently small $\delta$, since the height of $u$ is contracted for at least $1/7$ of the time and $(u,v)$ is not DC2 (if $v=(0,0)$, $(u,v)$ is either an asymptotic pair or $\frac{6}{7}\geq \Phi^*_{(u,v)}(\delta)$, since $u$ is recovered for at least $1/7$ of the time). Therefore $X$ has no $DC2$ pairs.\\

\paragraph{\bf Acknowledgment}
We wish to thank Professor Jaroslav Sm\' ital for his support.\\


\begin{thebibliography}{m}

\bibitem[1]{SchSm}{\sc Schweizer B., Sm\'ital J.}, {\em Measures of chaos and a spectral decomposition of dynamical systems on the interval\/}, Trans. Amer. Math. Soc. {\bf 344}, (1994), 737 -- 754.       
\bibitem[2]{BSS}{\sc Balibrea F., Sm\'ital J., \v Stef\'ankov\'a M.,} {\em The three versions of distributional chaos\/} Chaos Solitons Fractals {\bf 23} (2005), 1581--1583.
\bibitem[3]{LY}{\sc Li T., Yorke J.}, {\em Period three implies chaos\/}, Amer. Math. Monthly {\bf 82}, (1975), 985--992. 
\bibitem[4] {WLF} {\sc Wang H., Liao G.F., Fan Q.J.} {\em Substitution systems and the three versions of distributional chaos\/}, Topology Appl., {\bf 156} (2008), 26--267.
\bibitem[5]{WLW}{\sc Wang H., Lei F., Wang L.}, {\em DC3 and Li-Yorke chaos \/}, Appl. Math. Letters {\bf 31}, (2014), 29--33.
\bibitem[6]{O}{\sc Oprocha P.}, {\em Distributional chaos revisited\/}, Trans. Amer. Math. Soc. {\bf 361} (2009), 4901--4925. 
\bibitem[7]{DL}{\sc Downarowicz T., Lacroix Y.}, {\em Mesure-theoretic chaos\/}, Ergod. Th. Dynam. Sys. {\bf 34} (2014), 110--131.
\bibitem[8]{Down}{\sc Downarowicz T.,}{\em Survey of odometers and Toeplitz flows\/}, Algebraic \& topological dynamics, Contemp. Math. {\bf 385} (2005), 7--37.

\end{thebibliography}
\end{document}